\documentclass[12pt]{amsart}
\usepackage{graphicx} % Required for inserting images
\usepackage[english]{babel}
\usepackage{times}
\usepackage{amsmath,amsfonts,amssymb,amsopn,amscd,amsthm}
\usepackage{pdfpages}
\usepackage[french]{nomencl}
\usepackage{mathrsfs}
\usepackage{euscript}
\usepackage{graphicx,xspace}
\usepackage{epsfig}
\usepackage{enumerate} 
\usepackage{enumitem}
\usepackage{xcolor}
\usepackage[all]{xypic}
\usepackage{url}

\usepackage[T1]{fontenc}
\usepackage[utf8]{inputenc}
\usepackage{hyperref}
\usepackage{psfrag}
\usepackage{bm}
\usepackage{youngtab}
\usepackage{tikz}

\usetikzlibrary{snakes}
\usepackage[all]{xy}
\usepackage{varioref}
\theoremstyle{plain}
\newcounter{chapter}
\newtheorem{definition}{Definition}[chapter]
\newtheorem{theoreme}[definition]{Theorem}

\newtheorem{prop}[definition]{Proposition}
\newtheorem{ex}[definition]{Example}
\newtheorem{c-ex}[definition]{Counter-example}
\newtheorem{cor}[definition]{Corollary}
\newtheorem{lem}[definition]{Lemma}

\theoremstyle{definition}

\theoremstyle{remark}

\DeclareUnicodeCharacter{03C7}{\chi}
\DeclareUnicodeCharacter{B2}{^{2}}
\DeclareMathOperator{\supp}{supp}

\author[O. Tout]{Omar Tout}
\address{Department of Mathematics, College of Science, Sultan Qaboos University, P. O Box 36, Al Khod 123, Sultanate of Oman}
\email{o.tout@squ.edu.om}

\title{A note on the structure coefficients of the centraliser algebra}

\begin{document}

\maketitle

\begin{abstract} In this note we generalize the definition of partial permutations of Ivanov and Kerov and we build a universal algebra which projects onto the $m$-centraliser algebra defined by Creedon in \cite{Creedon}. We use it to present a new proof for the polynomiality property of the structure coefficients of the $m$-centraliser algebra and to obtain upper bounds for the polynomial degrees.
\end{abstract}

\section{Introduction}

Throughout this paper $m$ and $n$ will be integers satisfying $0\leq m\leq n.$ As usual $[n]$ will denote the set of integers $\lbrace 1,2,\cdots , n \rbrace$ and $[0]$ will be considered to be the emptyset $\emptyset.$ We denote by $\mathcal{S}_n$ the symmetric group on $[n]$ and by $Stab_n(m)$ the subgroup of $\mathcal{S}_n$ which contains all the permutations of $[n]$ which fixes all the elements of $[m].$ 

For $\ast$ a formal symbol, we define an \textit{$m$-marked cycle shape}\footnote{The definition in \cite{Creedon} does not allow cycles of form $(\ast).$} to be a finite collection of cycles with entries in $[m]\cup \lbrace \ast \rbrace$ such that the multiset of entries among the cycles equals $[m]\cup \lbrace \ast^r \rbrace$ for some $r\in \mathbb{Z}_{\geq 0},$ in particular each element of $[m]$ appears precisely once. The \textit{size} of an $m$-marked cycle shape $\rho,$ denoted $|\rho|,$ is the size of its multiset of entries and $m_1(\rho)$ will denote the number of cycles of $\rho$ of the form $(\ast)$ (cycles of form $(a),$ with $a\in [m],$ are not counted). For example, if $\sigma=(1,\ast,3)(2)(4,5)(\ast)(\ast,\ast)$ then $|\sigma|=9$ and $m_1(\sigma)=1.$ The set of $m$-marked cycle shapes with size $n$ will be denoted $\mathcal{P}^m_{n}.$ The \textit{$m$-marked cycle type} of a permutation of $n$ is the $m$-marked cycle shape obtained from its cycle notation by replacing all the integers in $[n]\setminus [m]$ with $\ast$ symbol. For example, the $2$-marked cycle type of $(1,3,5)(2)(4,7)(6)$ is $(1,\ast,\ast)(2)(\ast,\ast)(\ast).$

In \cite{Creedon}, Creedon defines the \textit{$m$-centraliser algebra} as 
\begin{equation}
    Z_{n,m}:=\lbrace z\in \mathbb{Z}[\mathcal{S}_n] \mid \tau z=z \tau \text{ for all } \tau\in Stab_n(m) \rbrace.
\end{equation}
The set of $m$-class sums $\lbrace K_\lambda \mid \lambda\in \mathcal{P}^m_{n} \rbrace,$ where $K_\lambda$ is the formal sum of all permutations of $[n]$ with $m$-marked cycle shape $\lambda,$ is a $\mathbb{Z}$-basis for $Z_{n,m}.$ As such, for any $\lambda,\delta \in \mathcal{P}^m_{n},$ we have
\begin{equation}\label{MainEq}
    K_\lambda K_\delta=\sum_{\rho \in \mathcal{P}^m_{n}}c_{\lambda,\delta}^\rho(n)K_\rho
\end{equation}
with structure coefficients $c_{\lambda,\delta}^\rho(n)\in \mathbb{N}\cup \lbrace 0\rbrace.$ When $m=0,$ $Stab_n(0)=\mathcal{S}_n,$ $\mathcal{P}^0_{n}$ may be identified with the set of all partitions of $n,$ and $Z_{n,0}$ is the center of the symmetric group algebra. A proof of the polynomiality property for the structure coefficients of $Z_{n,0}$ first appears in \cite{FarahatHigman}. In \cite{IvanovKerov}, by introducing partial permutations, Ivanov and Kerov give a combinatorial proof for the polynomiality property. In \cite[Theorem 3.0.3]{Creedon}, Creedon shows, using the Farahat and Higman approach, that the structure coefficients $c_{\lambda,\delta}^\rho(n)$ of $Z_{n,m}$ are polynomials in $n.$ The goal of this paper is to develop an approach similar to that of Ivanov and Kerov in order to prove Creedon's Theorem. 

\section{The semi-group of $m$-partial permutations}

We define an \textit{$m$-partial permutation} of $[n]$ to be a pair $(d,\omega)$ where $[m]\subseteq d\subseteq [n]$ and $\omega \in \mathcal{S}_d.$ Let $\mathfrak{P}^m_n$ denote the set of all $m$-partial permutations of $[n].$ We have:
\begin{equation*}
\vert \mathfrak{P}^m_n\vert=\sum_{k=0}^{n-m}{n-m \choose k}(m+k)!.
\end{equation*}

If $(d,\omega)\in \mathfrak{P}^m_n,$ we denote by $\widetilde{\omega}$ the permutation of $[n]$ defined as follows:
$$\widetilde{\omega}(a)=
\left\{
\begin{array}{ll}
  \omega(a) & \qquad \mathrm{if}\quad a\in d, \\
  a & \qquad \mathrm{if}\quad a\in [n]\setminus d. \\
 \end{array}
 \right.$$
The product of two elements $(d_1,\omega_1)$ and $(d_2,\omega_2)$ of $\mathfrak{P}^m_n,$ is defined as follows:
\begin{equation*}(d_1,\omega_1)\cdot (d_2,\omega_2)=(d_1\cup d_2,\widetilde{\omega}_{1_{|d_1\cup d_2}}\circ \widetilde{\omega}_{2_{|d_1\cup d_2}}),\end{equation*}
where $\widetilde{\omega}_{1_{|d_1\cup d_2}}$ and $\widetilde{\omega}_{2_{|d_1\cup d_2}}$ denote respectively the restrictions of $\widetilde{\omega}_1$ and $\widetilde{\omega}_2$ on $d_1\cup d_2.$ This product gives $\mathfrak{P}^m_n$ a semi-group structure with identity element $([m],e_{[m]}),$ where $e_{[m]}$ is the identity permutation of $[m].$

\subsubsection{Conjugacy classes of $\mathfrak{P}^m_n$}

It would be clear that the definition of $m$-marked cycle type for permutations of $\mathcal{S}_n$ can be extended naturally to $m$-partial permutations of $[n].$ For example, the $5$-marked cycle type of $(\{ 1,2,3,4,5,7,9,10,15 \} , (1,9)(3,7,10)(5,15,4)(2)(3))$ is $$(1,\ast)(3,\ast,\ast)(5,\ast,4)(2)(3).$$
Let $Stab_n(m)$ act on $\mathfrak{P}^m_n$ by:
$$\sigma\cdot (d,\omega)=(\sigma(d),\sigma\circ \omega\circ \sigma^{-1})$$ 
for any $\sigma\in Stab_n(m)$ and $(d,\omega)\in \mathfrak{P}^m_n.$ The orbits of this action will be denoted the conjugacy-classes of $\mathfrak{P}^m_n.$ Two $m$-partial permutations $(d_1,\omega_1)$ and $(d_2,\omega_2)$ belong to the same orbit if and only if there exists $\sigma\in Stab_n(m)$ such that $(d_2,\omega_2)=(\sigma(d_1),\sigma\circ \omega_1\circ \sigma^{-1}),$ that is $d_2=\sigma(d_1)$ and $\omega_2=\sigma\circ \omega_1\circ \sigma^{-1}.$ Therefore, $(d_1,\omega_1)$ and $(d_2,\omega_2)$ belong to the same orbit if and only if they both have the same $m$-marked cycle type. Thus, the conjugacy classes of $\mathfrak{P}^m_n$ can be indexed by the elements of $\mathcal{P}^m_{\leq n},$ the set of $m$-marked cycle shapes with size at most $n.$ For an $m$-marked cycle shape $\rho\in \mathcal{P}^m_{\leq n},$ the conjugacy-class of  $\mathfrak{P}^m_n$ associated to $\rho$ is denoted by $A_\rho(n)$ and defined as follows:
$$A_\rho(n):=\lbrace (d,\omega)\in \mathfrak{P}^m_n \vert\text{$\rho$ is the $m$-marked cycle type of $(d,\omega)$} \rbrace.$$

For $\rho\in \mathcal{P}^m_{\leq n},$ we define $\underline{\rho}_n$ to be the $m$-marked cycle shape of size $n$ obtained from $\rho$ by adding $n-|\rho|$ cycles of form $(\ast)$ to $\rho.$ For example, if $m=3,$ $\lambda=(1,\ast)(2)(3)$ and $\delta=(2,\ast)(1,\ast,3),$ then $\underline{\lambda}_7=(1,\ast)(2)(3)(\ast)(\ast)(\ast)$ and $\underline{\delta}_7=(2,\ast)(1,\ast,3)(\ast)(\ast).$ To simplify notation, an element $([n],f)$ in $A_{\underline{\rho}_n}(n)$ will be represented only by $f.$

\begin{prop}\label{prop:taille_A_rho_n}
Let $\rho\in \mathcal{P}^m_{\leq n},$ we have :
$$\vert A_{\rho}(n)\vert=\begin{pmatrix}
n-|\rho|+m_1(\rho) \\
m_1(\rho)
\end{pmatrix}\cdot \vert A_{\underline{\rho}_n}(n)\vert.$$
\end{prop}
\begin{proof} 
Consider the following map $$\begin{array}{ccccc}
\Theta & : & A_{\rho}(n) & \to & A_{\underline{\rho}_n}(n) \\
& & (d,\omega) & \mapsto & \widetilde{\omega}. \\
\end{array}$$
For any $v,v^{'}\in A_{\underline{\rho}_n}(n)$ with $v\neq v^{'},$ we have $\Theta^{-1}(v)\cap \Theta^{-1}(v^{'})=\emptyset,$ therefore $$\vert A_{\rho}(n)\vert=\sum_{v\in A_{\underline{\rho}_n}(n)}\vert \Theta^{-1}(v)\vert.$$
For $v\in A_{\underline{\rho}_n}(n),$ denote by $\supp^m(v)$ the following set:
$$\supp^m(v)=[m]\cup \lbrace x\in [n]~\vert~ v(x)\neq x\rbrace.$$
Since the $m$-marked cycle type of $v$ is $\underline{\rho}_n,$ we have $\vert \supp^m(v)\vert=\vert\rho\vert-m_1(\rho)$. But
$$\Theta^{-1}(v)=\lbrace (d,\omega)\in A_{\rho}(n) \text{ such that } \widetilde{\omega}=v\rbrace.$$ 
To build a partial permutation $(d,\omega)\in \Theta^{-1}(v),$ $\omega$ should coincide with $v$ on $\supp^m(v)$ (for this to happen, $\supp^m(v)$ should be a subset of $d$) and we choose $m_1(\rho)=\vert d\vert-\vert \supp^m(v)\vert$ fixed elements for $\omega$ from among the $n-|\rho|+m_1(\rho)$ fixed elements of $v.$ Therefore, $\vert \Theta^{-1}(v)\vert=\begin{pmatrix}
n-|\rho|+m_1(\rho) \\
m_1(\rho)
\end{pmatrix}$.
\end{proof}

\begin{ex}
Let $\lambda=(1,\ast,\ast)(2)(\ast),$ then $|\lambda|=5,$ $m_1(\lambda)=1,$
$$A_\lambda (6)=\lbrace (1,3,4)(2)(5); (1,3,4)(2)(6);
(1,3,5)(2)(4); (1,3,5)(2)(6);$$
$$(1,3,6)(2)(4); (1,3,6)(2)(5);
(1,4,5)(2)(3); (1,4,5)(3)(6);$$
$$(1,4,6)(2)(3); (1,4,6)(2)(5);
(1,4,3)(2)(5); (1,4,3)(2)(6);$$
$$(1,5,3)(2)(4); (1,5,3)(2)(6);
(1,5,4)(2)(3); (1,5,4)(2)(6);$$
$$(1,5,6)(2)(3); (1,5,6)(2)(4);
(1,6,3)(2)(4); (1,6,3)(2)(5);$$
$$(1,6,4)(2)(3); (1,6,4)(2)(5);
(1,6,5)(2)(3); (1,6,5)(2)(4)
\rbrace$$
and 
$$A_{\underline{\lambda}_6} (6)=\lbrace (1,3,4)(2)(5)(6); 
(1,3,5)(2)(4)(6); $$
$$(1,3,6)(2)(4)(5); 
(1,4,5)(2)(3)(6); $$
$$(1,4,6)(2)(3)(5);
(1,4,3)(2)(5)(6);$$
$$(1,5,3)(2)(4)(6);
(1,5,4)(2)(3)(6);$$
$$(1,5,6)(2)(3)(4);
(1,6,3)(2)(4)(5);$$
$$(1,6,4)(2)(3)(5);
(1,6,5)(2)(3)(4)
\rbrace$$

It is easy then to verify that 
$$24=\vert A_{\lambda}(6)\vert=\begin{pmatrix}
6-|\lambda|+m_1(\lambda) \\
m_1(\lambda)
\end{pmatrix}\cdot \vert A_{\underline{\lambda}_6}(6)\vert=\begin{pmatrix}
6-5+1 \\
1
\end{pmatrix}\cdot 12.$$
\end{ex}

Let us extend by linearity the action of $Stab_n(m)$ on $\mathfrak{P}^m_n$ to $\mathcal{B}^m_n:=\mathbb{C}[\mathfrak{P}^m_n],$ the semi-group algebra of $\mathfrak{P}^m_n,$ and consider 
$$\mathcal{A}^m_n={(\mathcal{B}^m_n)}^{Stab_n(m)}:=\lbrace b\in \mathcal{B}^m_n \text{ such that for any $\sigma\in Stab_n(m),$ } \sigma\cdot b=b \rbrace,$$ the sub-algebra of invariant elements of $\mathcal{B}^m_n.$
The surjective homomorphism $$\begin{array}{ccccc}
\psi & : & \mathfrak{P}^m_n & \to & \mathcal{S}_n \\
& & (d,\omega) & \mapsto & \widetilde{\omega} \\
\end{array}$$
can be linearly extended to a surjective algebras homomorphism
$\psi:\mathcal{B}^m_n\rightarrow \mathbb{C}[\mathcal{S}_n]$.
%L'écriture générique d'un élément $b\in \mathcal{B}_n$ est :
%$$b=\sum_{k=0}^{n}\sum_{\vert d\vert=k}\sum_{\omega\in \mathcal{S}_d}b_{d,\omega}(d,\omega),$$
%où les $b_{d,\omega}$ sont des nombres complexes.
For any $\sigma\in Stab_n(m)$ and any $b\in \mathcal{B}^m_n,$ we have:
$$\psi (\sigma\cdot b)=\sigma\cdot \psi(b)$$
(this is due to the fact that $\widetilde{\sigma^{-1} \omega\sigma}=\sigma^{-1} \widetilde{\omega}\sigma$).Therefore
$$\psi(\mathcal{A}^m_n)=(\psi(\mathcal{B}^m_n))^{Stab_n(m)}=Z_{n,m}.$$
For $\rho\in \mathcal{P}^m_{\leq n},$ define
$${\bf A}_{\rho}(n):=\sum_{(d,\omega)\in A_{\rho}(n)}(d,\omega).$$
The elements $({\bf A}_\rho(n))_{\rho \in \mathcal{P}^m_{\leq n}}$ form a basis for $\mathcal{A}^m_n.$ One of the important properties of $\psi$ is that for any $\rho \in \mathcal{P}^m_{\leq n},$ we have:
\begin{equation}\label{imagepsi}
\psi({\bf A}_\rho(n))=\begin{pmatrix}
n-|\rho|+m_1(\rho) \\
m_1(\rho)
\end{pmatrix} K_{\underline{\rho}_n}.
\end{equation}
This equation is a direct consequence of Proposition \ref{prop:taille_A_rho_n} and the fact that ${\bf A}_{\underline{\rho}_n}(n)= K_{\underline{\rho}_n}.$

\subsubsection{The algebras $\mathcal{A}^m_\infty$ and $\mathcal{B}^m_{\infty}$}
Consider the family $(\mathcal{B}^m_n, \varphi^m_n)$ where $\varphi^m_n:\mathcal{B}^m_{n+1}\rightarrow \mathcal{B}^m_n$ is the homomorphism defined as follows:
$$\varphi^m_n(d,\omega)=
\left\{
\begin{array}{ll}
  (d,\omega) & \qquad \mathrm{if}\quad d\subset [n]; \\
  0 & \qquad \mathrm{otherwise.}\quad \\
 \end{array}
 \right.$$
Let us denote by $\mathcal{B}^m_\infty$ the projective limit of the algebras $\mathcal{B}^m_n$ with respect to $\varphi^m_n.$
\begin{definition}
An $m$-partial permutation of $\mathbb{N}$ is a pair $(d,\omega)$ where $[m]\subseteq d\subsetneq \mathbb{N}$ is a finite subset of $\mathbb{N}$ and $\omega\in \mathcal{S}_d.$
\end{definition} 
We denote by $\mathfrak{P}^m$ the set of all partial $m$-permutations of $\mathbb{N}.$ An element $b\in \mathcal{B}^m_\infty$ is canonically written as follows :
$$b=\sum_{k=m}^{\infty}\sum_{[m]\subseteq d, \atop{\vert d\vert=k}}\sum_{\omega\in \mathcal{S}_d}b_{d,\omega}(d,\omega),$$
where $b_{d,\omega}$ are complex numbers.

Consider $\mathcal{S}_\mathbb{N},$ the group of permutations of $\mathbb{N}$ with finite support and let $Stab_\mathbb{N}(m)$  be its subgroup consisting of the permutations which fix each element of $[m].$ As previously, $Stab_\mathbb{N}(m)$ acts on $\mathcal{B}^m_\infty$ and we consider $\mathcal{A}^m_\infty=(\mathcal{B}^m_\infty)^{Stab_\mathbb{N}(m)}$ the sub-algebra of invariant elements of $\mathcal{B}^m_\infty$ under the action of $Stab_\mathbb{N}(m).$ An element $b\in \mathcal{B}^m_\infty$ belongs to $\mathcal{A}^m_\infty$ if and only if :
$$b_{d,\omega}=b_{\sigma (d),\sigma\circ\omega\circ \sigma^{-1}},~\forall \sigma\in Stab_\mathbb{N}(m).$$
For an $m$-marked cycle shape $\rho\in \mathcal{P}^m,$ define $A_\rho$ as follows :
$$A_\rho:=\lbrace (d,\omega)\in \mathfrak{P}^m\text{ such that $\rho$ is the $m$-marked cycle type of } (d,\omega) \rbrace.$$
%Soit $b=\sum_{{\overset{d\subseteq \mathbb{N}}{\omega\in \mathcal{S}_{d}}}}b_{d,\omega}(d,\omega)\in \mathcal{B}_\infty$, on a :
%$$b\in \mathcal{A}_\infty \Leftrightarrow b_{d,\omega}=b_{\sigma (d),\sigma\circ\omega\circ \sigma^{-1}},~\forall \sigma\in \mathcal{S}_\infty.$$ 
Any element of $\mathcal{A}^m_\infty$ can be written as an infinite linear combination of  $({\bf A}_\rho)_{\rho \in \mathcal{P}^m},$ where
$${\bf A}_\rho=\sum_{(d,\omega)\in A_\rho}(d,\omega).$$
If $\lambda$ and $\delta$ are two elements of $\mathcal{P}^m,$ the structure coefficients $k_{\lambda\delta}^\rho$ of the algebra $\mathcal{A}^m_\infty$ are defined by the following equation:
\begin{equation}\label{eq:str_coef_A_infini}
{\bf A}_\lambda{\bf A}_\delta=\sum_{\rho \in \mathcal{P}^m}k_{\lambda\delta}^\rho{\bf A}_\rho.
\end{equation}
Let us denote by $\theta^m_n$ the natural homomorphism between $\mathcal{B}^m_\infty$ and $\mathcal{B}^m_n,$
%$$\begin{array}{ccccc}
%\theta_n & : & \mathcal{B}_\infty & \to & \mathcal{B}_n \\
%& & (b_n)_{n\geqslant 1} & \mapsto & b_n \\
%\end{array}$$
%A des isomorphismes près, on peut voir $\theta_n$ comme la morphisme suivante:
$$\begin{array}{ccccc}
\theta^m_n & : & \mathcal{B}^m_\infty & \to & \mathcal{B}^m_n \\
& & \sum_{{\overset{[m]\subseteq d\subsetneq \mathbb{N}}{\omega\in \mathcal{S}_{d}}}}b_{d,\omega}(d,\omega) & \mapsto & \sum_{{\overset{[m]\subseteq d\subseteq [n]}{\omega\in \mathcal{S}_{d}}}}b_{d,\omega}(d,\omega) \\
\end{array}.$$
Remark that: $$\theta^m_n({\bf A}_\rho)={\bf A}_{\rho}(n).$$

As a consequence of Equation \eqref{eq:str_coef_A_infini}, $k_{\lambda\delta}^{\rho}$ are also the structure coefficients of the algebra $\mathcal{A}^m_n$ :
\begin{equation*}
{\bf A}_\lambda(n){\bf A}_\delta(n)=\sum_{\rho\in \mathcal{P}^m_{\leq n}}k_{\lambda\delta}^{\rho}{\bf A}_\rho(n),
\end{equation*}
where $\lambda$ and $\delta$ are two $m$-marked cycle shapes of size less than or equal to $n.$ In particular, the structure coefficients $k_{\lambda\delta}^{\rho}$ of $\mathcal{A}^m_\infty$ are independent of $n,$ and are the structure coefficients of $\mathcal{A}^m_n$ for any $n.$

Let $\rho$ be an $m$-marked cycle shape of size less than or equal to $n.$ Fix an $m$-partial permutation $(d,\omega)$ of $n$ with $\rho$ as an $m$-marked cycle type. The coefficient $k_{\lambda\delta}^\rho$ is the cardinal of the set $E_{\lambda\delta}^\rho(n)$ defined as follows:
\begin{equation*}
E_{\lambda\delta}^\rho(n)=\lbrace((d_1,\omega_1),(d_2,\omega_2))\in A_{\lambda}(n)\times A_{\delta}(n)~\mid~(d_1,\omega_1)\cdot (d_2,\omega_2)=(d,\omega)\rbrace.
\end{equation*}

\begin{cor}\label{deg1}
If $k_{\lambda\delta}^\rho\neq 0,$ then $\vert \rho\vert\leq\vert \lambda\vert+\vert\delta \vert.$
\end{cor}
\begin{proof}
If $k_{\lambda\delta}^\rho\neq 0,$ we have $E_{\lambda\delta}^\rho(n)\neq \emptyset\Rightarrow \exists d_1,d_2\subset [n]$ such that $\vert d_1\vert=\vert\lambda\vert,\vert d_2\vert=\vert\delta\vert$ and $\vert\rho\vert=\vert d_1\cup d_2\vert\leq \vert d_1\vert+\vert d_2\vert=\vert\lambda\vert+\vert\delta\vert.$
\end{proof}

\begin{cor}\label{cor:sous_alg_de_A_infty}
The set of finite linear combinations of (${\bf{A}}_\lambda$), denoted by $\widetilde{\mathcal{A}^m_\infty},$ is a sub-algebra of $\mathcal{A}^m_\infty.$ The family $({\bf{A}}_\lambda)_{\lambda\in \mathcal{P}^m}$ forms a basis for $\widetilde{\mathcal{A}^m_\infty}.$
\end{cor}

\subsubsection{A proof of the polynomiality property}

\begin{definition}
    An $m$-marked cycle shape $\rho$ is said to be proper if $m_1(\rho)=0.$ The set of all proper $m$-marked cycle shape will be denoted $\mathcal{PP}^m.$
\end{definition}

If $\rho$ is an $m$-marked cycle shape of size less than or equal to $n,$ then for any $0\leq k\leq n-|\rho|,$ we define $\rho^k$ to be the following $m$-marked cycle shape of size $|\rho|+k:$
$$\rho^k:=\rho \underbrace{(\ast)(\ast)\cdots (\ast)}_{\text{$k$ times}}.$$
It would be clear that $m_1(\rho^k)=m_1(\rho)+k.$

\begin{lem}\label{Lemma}
    If $\rho$ is an $m$-marked cycle shape of size less than or equal to $n,$ then for any $0\leq k\leq n-|\rho|,$ we have
    \begin{equation*}
K_{\underline{\rho}_n}=K_{\underline{\rho^k}_n}.
\end{equation*}
\end{lem}

\begin{theoreme}
Let $\lambda, \delta$ and $\rho$ be three proper $m$-marked cycle shapes and let $n\geq m$ be a natural number satisfying $n\geq |\lambda|, |\delta|,|\rho|.$ The structure coefficients $c_{\underline{\lambda}_n\underline{\delta}_n}^{\underline{\rho}_n}(n)$ defined by Equation \eqref{MainEq} can be written as follows :
\begin{equation*}
c_{\underline{\lambda}_n\underline{\delta}_n}^{\underline{\rho}_n}(n)=\sum_{k=0}^{n-\vert\rho\vert}k_{\lambda\delta}^{\rho^k}\begin{pmatrix}
n-\vert\rho\vert \\
k
\end{pmatrix}, 
\end{equation*}
where $k_{\lambda\delta}^{\rho^k}$ are integers independent of $n.$
\end{theoreme}
\begin{proof}
By Equation \eqref{imagepsi}, we have : $\psi({\bf A}_\lambda(n))=K_{\underline{\lambda}_n}$ and $\psi({\bf A}_\delta(n))=K_{\underline{\delta}_n}.$
Consider in $\mathcal{A}^m_n$ the following equation:
\begin{equation*}
{\bf A}_\lambda(n){\bf A}_\delta(n)=\sum_{\tau\in \mathcal{P}^m_{\leq n}\atop{ |\tau|\leq |\lambda|+|\delta|}}k_{\lambda\delta}^{\tau}{\bf A}_\tau(n).
\end{equation*}
By applying $\psi,$ we get :
\begin{equation*}
K_{\underline{\lambda}_n}K_{\underline{\delta}_n}=\sum_{\tau\in \mathcal{P}^m_{\leq n}\atop{ |\tau|\leq |\lambda|+|\delta|}}k_{\lambda\delta}^{\tau}\begin{pmatrix}
n-|\tau|+m_1(\tau) \\
m_1(\tau)
\end{pmatrix}K_{\underline{\tau}_n}.
\end{equation*}

Therefore, by using Lemma \ref{Lemma},the right hand side of the above equation can be written
$$\sum_{\rho \in \mathcal{PP}^m_{\leq n}\atop{|\rho|\leq |\lambda|+|\delta|}} \left[\sum_{k=0}^{n-\vert\rho\vert}k_{\lambda\delta}^{\rho^k}\begin{pmatrix}
n-\vert\rho^k\vert+m_1(\rho^k) \\
m_1(\rho^k)
\end{pmatrix}\right] K_{\underline{\rho}_n}.$$
After simplification, we get :
$$\sum_{\rho \in \mathcal{PP}^m_{\leq n}\atop{|\rho|\leq |\lambda|+|\delta|}} \left[\sum_{k=0}^{n-\vert\rho\vert}k_{\lambda\delta}^{\rho^k}\begin{pmatrix}
n-\vert\rho\vert \\
k
\end{pmatrix} \right] K_{\underline{\rho}_n},$$
and the result follows.
\end{proof}
\begin{cor}\label{cor:deg_cof_str_S_n}
Let $\lambda$ and $\delta$ be two proper $m$-marked cycle shapes. For any proper $m$-marked cycle shape $\rho$ of size less than or equal to $|\lambda|+|\delta|$ and for any $n\geq |\lambda|,|\delta|,|\rho|$ we have:
\begin{equation*}
\deg(c_{\underline{\lambda}_n\underline{\delta}_n}^{\underline{\rho}_n}(n))=\max_{\overset{0\leq k\leq n-\vert\rho\vert}{k_{\lambda\delta}^{\rho^k}\neq 0}}\underbrace{\deg\left( \begin{pmatrix}
n-\vert\rho\vert \\
k
\end{pmatrix} \right) }_{k}=\displaystyle \max_{\overset{0\leq k\leq n-\vert\rho\vert}{k_{\lambda\delta}^{\rho^k}\neq 0}}k.
\end{equation*}
\end{cor}

\begin{ex}
    Let $\lambda = (1,2,\ast)$ and $\delta = (1,\ast)(2,\ast)(\ast,\ast).$ We have:
    \begin{equation*}
{\bf A}_\lambda{\bf A}_\delta= {\bf A}_{(1)(2,\ast,\ast)(\ast,\ast)} + {\bf A}_{(1,\ast,2)(\ast,\ast)(\ast)} + {\bf A}_{(1,\ast,2,\ast,\ast,\ast)}+{\bf A}_{(1,\ast,2,\ast,\ast)(\ast,\ast)}.
\end{equation*}
For $n\geq 7,$ we have:
$${\bf A}_\lambda(n){\bf A}_\delta(n)={\bf A}_{(1)(2,\ast,\ast)(\ast,\ast)}(n) + {\bf A}_{(1,\ast,2)(\ast,\ast)(\ast)}(n) + {\bf A}_{(1,\ast,2,\ast,\ast,\ast)}(n)$$
$$+{\bf A}_{(1,\ast,2,\ast,\ast)(\ast,\ast)}(n).$$
and after applying $\psi$ we obtain the following equation:
$$K_{\underline{\lambda}_n}K_{\underline{\delta}_n}=K_{\underline{(1)(2,\ast,\ast)(\ast,\ast)}_n} + (n-5) K_{\underline{(1,\ast,2)(\ast,\ast)}_n} + K_{\underline{(1,\ast,2,\ast,\ast,\ast)}_n}$$
$$+ K_{\underline{(1,\ast,2,\ast,\ast)(\ast,\ast)}_n}.$$
\end{ex}

\begin{ex}\label{ex:main}
    Let $\lambda = (1,2)(\ast,\ast)$ and $\delta = (1)(2)(\ast,\ast).$ We have:
    \begin{equation*}
{\bf A}_\lambda{\bf A}_\delta= 2{\bf A}_{(1,2)(\ast,\ast)(\ast,\ast)} + 3{\bf A}_{(1,2)(\ast,\ast,\ast)} + {\bf A}_{(1,2)(\ast)(\ast)}.
\end{equation*}
For $n\geq 6,$ we have:
\begin{equation*}
{\bf A}_\lambda(n){\bf A}_\delta(n)=2{\bf A}_{(1,2)(\ast,\ast)(\ast,\ast)}(n) + 3{\bf A}_{(1,2)(\ast,\ast,\ast)}(n) + {\bf A}_{(1,2)(\ast)(\ast)}(n)
\end{equation*}
and after applying $\psi$ we obtain the following equation:
\begin{eqnarray*}
K_{\underline{\lambda}_n}K_{\underline{\delta}_n}&=&2K_{\underline{(1,2)(\ast,\ast)(\ast,\ast)}_n} + 3K_{\underline{(1,2)(\ast,\ast,\ast)
}_n} + {n-4+2 \choose 2} K_{\underline{(1,2)(\ast)(\ast)}_n}\\
&=&2K_{\underline{(1,2)(\ast,\ast)(\ast,\ast)}_n} + 3K_{\underline{(1,2)(\ast,\ast,\ast)
}_n} + \frac{(n-2)(n-3)}{2} K_{\underline{(1,2)}_n}
\end{eqnarray*}

\end{ex}

\subsection{Filtrations on $\mathcal{A}^m_\infty$}\label{sec:filt}

By Corollary \ref{deg1}, it would be clear that
$$\deg_1 ({\bf A}_\rho) = |\rho|,$$
defines a filtration on $\mathcal{A}^m_\infty.$ Inspired by \cite[Proposition 10.1]{IvanovKerov}, one can show the following proposition. 

\begin{prop}\label{filt_deg_2}
The function $\deg_{2}$ defined by:
$$\deg_2(d,\omega)=\vert d \vert+\vert m_1(\omega)\vert,$$
where $m_1(\omega)$ is the set of fixed points of $\omega$ which do not belong to $[m]$ is a filtration.
\end{prop}
\begin{proof}
    The proof is similar to the proof of Proposition 10.1 in \cite{IvanovKerov}.
\end{proof}

If $\deg$ is a filtration on $\mathcal{A}^m_\infty$, then the equation
$${\bf A}_\lambda {\bf A}_\delta=\sum_\rho k_{\lambda\delta}^\rho {\bf A}_\rho,$$
implies  
\begin{equation}\label{eq:filtr}
\deg({\bf A}_\lambda {\bf A}_\delta):=\max_{\rho, k_{\lambda\delta}^\rho\neq 0} \deg({\bf A}_\rho)\leq \deg({\bf A}_\lambda)+ \deg ({\bf A}_\delta).
\end{equation}
\begin{cor}\label{prop:borne}
If $\lambda, \delta$ and $\rho$ are three proper $m$-marked cycle shapes and $n\geq m$ is a natural number satisfying $n\geq |\lambda|, |\delta|,|\rho|,$ then 
$$\deg(c_{\underline{\lambda}_n\underline{\delta}_n}^{\underline{\rho}_n}(n))\leq \frac{|\lambda|+|\delta|-|\rho|}{2}.$$
\end{cor}
\begin{proof}
By Corollary \ref{cor:deg_cof_str_S_n}, we have :
$$\deg(c_{\underline{\lambda}_n\underline{\delta}_n}^{\underline{\rho}_n}(n))=\max_{\overset{0\leq k\leq n-\vert\rho\vert}{k_{\lambda\delta}^{\rho^k}\neq 0}}k.$$
Using the filtration $\deg_2,$ Equation \eqref{eq:filtr} gives:
$$\deg_2({\bf A}_{\rho^k})\leq \max_{\rho, k_{\lambda\delta}^\rho\neq 0} \deg_2({\bf A}_\rho)\leq \deg_2({\bf A}_\lambda)+ \deg_2 ({\bf A}_\delta),$$
for any $k\geq 0$ satisfying $k_{\lambda\delta}^{\rho^k}\neq 0.$ Therefore, for any $k\geq 0$ satisfying $k_{\lambda\delta}^{\rho^k}\neq 0,$ we have:
$$|\rho|+2k\leq |\lambda|+|\delta|.$$
The result follows.
\end{proof}

\begin{ex}
    For $\lambda=(1,2)(\ast,\ast),$ $\delta=(1)(2)(\ast,\ast)$ and $\rho=(1,2),$ using Corollary \ref{prop:borne}, we obtain that the degree of $c_{\underline{\lambda}_n\underline{\delta}_n}^{\underline{\rho}_n}(n)$ is bounded by $3.$ The degree of $c_{\underline{\lambda}_n\underline{\delta}_n}^{\underline{\rho}_n}(n)$ is $2$ as obtained in Example \ref{ex:main}. For $\mu=(1,2,\ast,\ast),$ $\gamma=(1)(2,\ast,\ast)$ and $\nu=(1,2,\ast)(\ast,\ast),$ using Corollary \ref{prop:borne}, we obtain that the degree of $c_{\underline{\mu}_n\underline{\gamma}_n}^{\underline{\nu}_n}(n)$ is bounded by $1.$
\end{ex}

\end{document}